\documentclass[11pt]{article}
\usepackage{amssymb,amsmath,amsfonts,amsthm,graphicx,color}
\usepackage{bm}
\usepackage{tikz-cd}
\usepackage[all]{xy}
\usepackage{multicol}
\usepackage[colorlinks]{hyperref}
\usepackage{hyperref}

\usepackage{mathrsfs}
\usepackage{comment}
\usepackage{fullpage}
\usepackage{baskervald}
\usepackage[affil-it]{authblk}

\usepackage[scr=boondox]{mathalfa}
\usepackage{mathtools}

\DeclareMathSymbol{:}{\mathpunct}{operators}{"3A}

\renewcommand{\phi}{\varphi}

\DeclareMathOperator{\Art}{Art}

\DeclareMathOperator{\chern}{ch}
\DeclareMathOperator{\Td}{Td}
\DeclareMathOperator{\rk}{rk}

\newcommand{\m}{\bm{\mu}}
\newcommand{\B}{\bm{\mathrm B}}

\DeclareMathOperator{\Aut}{Aut}
\DeclareMathOperator{\End}{End}

\newcommand{\Q}{\mathbf{Q}}

\def\iso{\xrightarrow{\raisebox{-1pt}{\tiny$\sim$}}}

\newcommand\ms[1]{\mathscr{#1}}
\renewcommand{\H}{\operatorname{H}}
\newcommand{\tensor}{\otimes}
\DeclareMathOperator{\ob}{ob}

\newcommand{\G}{\bm{\mathrm G}}

\newcommand{\ltensor}{\stackrel{\mathbf{L}}{\otimes}}
\DeclareMathOperator{\LL}{\mathbf{L}}
\DeclareMathOperator{\R}{\mathbf{R}}
\newcommand{\pr}{\operatorname{pr}}

\newcommand{\Qcoh}{\operatorname{QCoh}}

\DeclareMathOperator{\Br}{Br}
\DeclareMathOperator{\BrSh}{\overline{Br}}
\DeclareMathOperator{\C}{\mathbf{C}}

\DeclareMathOperator{\Coh}{Coh}
\DeclareMathOperator{\Ext}{Ext}

\DeclareMathOperator{\id}{id}

\DeclareMathOperator{\Hom}{Hom}
\DeclareMathOperator{\Spec}{Spec}
\DeclareMathOperator{\Spf}{Spf}
\DeclareMathOperator{\Def}{Def}
 
\DeclareMathOperator{\pPerf}{\mathscr{P\!e\!r\!f}}
\DeclareMathOperator{\sPic}{\mathscr{P\!i\!c}}
\DeclareMathOperator{\Pic}{Pic}

\DeclareMathOperator{\Set}{Set}

\DeclareMathOperator{\tr}{tr}

\newcommand{\Z}{\mathbf{Z}}

\def\E{\mathscr{E}}

\def\oO{\mathscr{O}}

\def\X{\mathscr{X}}

\newtheorem{thm}{Theorem}

\makeatletter
\@addtoreset{thm}{section}
\@addtoreset{thm}{subsection}
\@addtoreset{thm}{subsubsection}
\@addtoreset{thm}{paragraph}
\makeatother

\renewcommand{\thethm}{\ifnum\value{subsubsection}>0{\thesubsubsection.\arabic{thm}}\else{\ifnum\value{subsection}>0{\thesubsection.\arabic{thm}}\else{\thesection.\arabic{thm}}\fi}\fi}

\numberwithin{equation}{subsection}

\newtheorem{lemma}[equation]{Lemma}

\newtheorem{lem}[equation]{Lemma}
\newtheorem{cor}[equation]{Corollary}
\newtheorem{prop}[equation]{Proposition}
\newtheorem{conj}[equation]{Conjecture}

\theoremstyle{remark}

\newtheorem{remark}[equation]{Remark}

\newtheorem{defn}[equation]{Definition}
\newtheorem{notation}[equation]{Notation}

\newtheorem{setting}[equation]{Setting}

\theoremstyle{definition}
\newtheorem{assumption}[equation]{Assumption}

\title{Fourier--Mukai partners of Enriques and bielliptic
  surfaces in positive characteristic}

\author{Katrina Honigs}
\affil{University of Utah}
\author{Max Lieblich} 
\affil{University of Washington}
\author[3,4]{Sofia Tirabassi}
\affil[3]{University of Bergen}
\affil[4]{Stockholm University}

\date{}

\begin{document}

\maketitle
\tableofcontents    

\section{Introduction}

A
\emph{Fourier--Mukai partner\/} of a smooth projective variety $X$
over a field $k$ is a smooth
projective variety $Y$ over $k$ such that there is a $k$-linear
equivalence of triangulated categories 
$D(X)\to D(Y)$, 
where $D(-)$ denotes the bounded derived category with coherent
cohomology (equivalently -- since $X$ and $Y$ are smooth -- perfect
complexes). 

A \emph{twisted variety\/} is a pair $(X,\alpha)$ with $X$
a variety and $\alpha\in\Br(X)$ a Brauer class. An isomorphism of
twisted varieties $(X,\alpha)\to(Y,\beta)$ is an isomorphism $f: X\to
Y$ such that $f^\ast\beta=\alpha$. 
Given a smooth projective twisted variety $(X,\alpha)$, we will let 
$D(X,\alpha)$ denote the bounded derived category of
$\alpha$-twisted $\ms O_X$-modules with coherent
cohomology.\footnote{There are some subtle foundational issues with this definition as it is written. See Section~\ref{sec:Definitions} for a discussion of ways in which it can be corrected.}
Another
pair $(Y,\beta)$ is a \emph{twisted Fourier--Mukai partner\/}
of $(X,\alpha)$ if there is a $k$-linear equivalence of triangulated
categories $D(X,\alpha)\to D(Y,\beta)$. We will say that $(X,\alpha)$
\emph{has no non-trivial twisted Fourier--Mukai partners\/} if any
twisted Fourier--Mukai partner
partner $(Y,\beta)$ of $(X,\alpha)$ is isomorphic to
$(X,\alpha)$.

In this note we prove the following results. 

\begin{thm}\label{thm:enriques}
  If $(X,\alpha)$ is a twisted Enriques surface over an algebraically closed field
  $k$ of characteristic at least $3$, then it does not have any
  nontrivial twisted Fourier--Mukai partners.
\end{thm}

\begin{thm}\label{thm:bielliptic}
  If $X$ is a bielliptic surface over an
  algebraically closed field $k$ of characteristic at least $5$,
  then it does not have any nontrivial (untwisted) Fourier--Mukai partners.
\end{thm}

These results were proven in the untwisted case over the
complex numbers by Bridgeland and Maciocia
\cite{bridgesurfaces}. Their argument uses Torelli theorems to reduce
the result to lattice-theoretic assertions, and thus relies heavily on
the complex numbers. The twisted cases of these results are proven over the
complex numbers by Addington and Wray with methods analogous to those
of Bridgeland and Maciocia, with the requisite modifications to
handle the twists.  

Our approach to proving this in positive characteristic is to lift
(twisted) derived equivalences to characteristic~$0$ and specialize isomorphisms using
the Matsusaka--Mumford theorem. The strategy is similar to that taken
in \cite{LO}. In the present situation, the central issue
 is the question of whether the fibers of the kernel,
viewed as complexes on $Y$, are smooth points in the moduli space of
complexes, something that is automatic for K3 surfaces
since obstruction spaces themselves vanish.
While the obstruction spaces for simple complexes on Enriques and
bielliptic surfaces are nonvanishing, we can
nevertheless show that the obstruction \emph{classes\/} vanish 
by using equivariant techniques similar to those originally used in
\cite{bridgesurfaces}.

In Section~\ref{sec:relat-bridg-maci} we briefly discuss the positive
characteristic version of the equivariance results of
\cite{bridgequotient}. In Section~\ref{sec.fi} we work through 
 producing an isomorphism
of deformation spaces that gives us a way to lift kernels. 
The extension of the ideas of \cite{deformingkernels} to the twisted setting, 
discussed in Section~\ref{sec:moduli-perf}, provides a framework for the argument.
In Section~\ref{sec.fo} we complete the proof of Theorems
\ref{thm:enriques} and \ref{thm:bielliptic}.  

\subsection*{Acknowledgments}
Honigs was partially supported by NSF MSPRF grant DMS-160628 and thanks Nick Addington and Karl Schwede for helpful conversations. Lieblich was
partially supported by NSF CAREER DMS-1056129, NSF standard grant
DMS-1600813, and a Simons Foundation Fellowship; he thanks Martin Olsson for helpful remarks. Tirabassi
was partially supported by grant 261756 of the Research Councils of
Norway; she is grateful to Christian Liedtke and Richard Thomas for
very interesting and stimulating mathematical discussions.

\section{Twisted sheaves}
\label{sec:twist-schem-some}

In this section we develop enough of the theory of twisted schemes and
moduli of perfect complexes of twisted sheaves to put our questions in
families. This is well-trodden in the literature (see e.g.~ \cite{caldararu_thesis,max_twisted,MR2306170} for twisted sheaves and \cite{deformingkernels} for moduli of perfect complexes), 
so we only
indicate the essentials here. The main distinction between our treatment here
and the existing literature is our careful attention to the distinction between
$\G_m$-twisting and $\m_n$-twisting. We will occasionally have to pass between
them, but we try to explain how the equivalences among the various models transfer.

\subsection{Rigidifications of twisted sheaves}
\label{sec:Definitions}

\begin{defn} 
  \begin{enumerate}
  \item A \emph{twisted scheme\/} is a pair $(X,\alpha)$ with $X$ a
    scheme and $\alpha\in\Br(X)$.
  \item A \emph{morphism of twisted schemes\/}
    $(X,\alpha)\to(Y,\beta)$ is a morphism $f: X\to Y$.
  \item The \emph{relative twisting class\/} of a morphism
    $f: (X,\alpha)\to (Y,\beta)$ is the class $f^\ast\beta-\alpha$.
  \item A morphism of twisted schemes is \emph{untwisted\/} if the relative twisting class is $0$. 
  \item An isomorphism of twisted schemes is an untwisted morphism of twisted schemes whose
    underlying morphism $f: X\to Y$ is a isomorphism.
  \end{enumerate}
\end{defn}

We will say that a morphism $f: (X,\alpha)\to (Y,\beta)$ has a property
$P$ (e.g., is \emph{flat\/},
\emph{proper\/}, \emph{locally of finite presentation\/}, etc.) if
$P$ holds for the underlying morphism $f: X\to Y$. 

\smallskip

Our purpose in defining twisted schemes is to study twisted sheaves.
In order to
properly handle the theory of determinants and Mukai vectors for different types
of twisting coefficients ($\G_m$ and $\m_n$), we will need to keep track of
several categories at once.

Fix a twisted scheme $(X,\alpha)$ and an integer $n$ invertible in the base.
Categories of ($n$-)twisted sheaves are
  not naturally associated to the pair $(X,\alpha)$, but rather to a
  ``rigidification'' of $\alpha$.
There are several ways to choose this rigidification 
(see \cite[\S3.1]{MR2388554} for more detail), which all yield equivalent categories of sheaves and, hence, derived categories of sheaves:

\begin{enumerate}
\item One type of rigidification uses a choice of \v{C}ech $2$-cocycle representatives for $\alpha$ \cite[Definition~1.2.1]{caldararu_thesis}
and common in the complex-analytic literature, but we will not
  use it here.
\item Choose an Azumaya algebra $\ms A$ on $X$ representing
  $\alpha$. An $n$-fold twisted sheaf can then be identified with an $\ms
  A^{\tensor n}$-module. 
\item Choose a closed subgroup $\G\subset\G_m$ and $\G$-gerbe $\ms X\to X$ whose
  associated Brauer class 
  is $\alpha$.
A locally free $\ms O_{\ms X}$-module $F$ has a natural sum decomposition by the weights of the right inertial action of $\G$
(cf.~\cite[\S 12.3]{olsson})
\begin{equation}\label{decomp}
F\simeq\bigoplus_{i\in\mathbb{Z}}F_i
\end{equation}
where we index the characters $\G\to \G_m$ via the surjection $\Z\simeq\Hom(\G_m,\G_m)\to\Hom(\G,\G_m)$:
the $i$-th component has  action given by scalar multiplication $i$ times. 
An $n$-fold twisted sheaf is a locally free $\ms O_{\ms X}$-module whose decomposition rests purely in weight $n$.
\end{enumerate}

\begin{remark}
In the second model, the notions of
flatness, quasi-coherence, coherence, homological dimension, etc., are
all equivalent to the same notions for the underlying $\ms
O_X$-module (since any Azumaya algebra is locally Morita-equivalent to
the structure sheaf). 
In the third model, those notions are the usual ones for
sheaves on a stack, and so we may seamlessly consider perfect complexes in either setting.

Unfortunately, the second realization does not preserve rank,
  determinant, etc, but the third is
 useful for studying
  determinants, ranks, Chern characters, etc., and so we work  with this realization.
\end{remark}

\begin{notation}
  We will write $\Coh^{(n)}(X,\alpha)$, resp. $\Qcoh^{(n)}(X,\alpha)$, resp.
  $D^{(n)}(X,\alpha)$ for the category of coherent $n$-fold twisted sheaves, resp.
  quasi-coherent $n$-fold twisted sheaves, resp. bounded derived category of
  $n$-fold twisted 
  sheaves with coherent cohomology. When we have a gerbe model $\ms X\to X$ of
  $\alpha$, we will write $\Qcoh^{(n)}(\ms X)$, etc., for the various categories
  of twisted sheaves.
\end{notation}

\subsection{Determinants}
In this section we discuss the determinant of a perfect complex of
twisted sheaves, using the gerbe definition of twisted sheaves.
For some closed subgroup $\G\subset\G_m$, with dual group
$\widehat \G:=\Hom(\G,\G_m)$, fix a $\G$-gerbe $\ms X\to X$. Let $\Pic(\ms X)$ be the group of invertible sheaves on $\X$.
The decomposition \eqref{decomp} shows the elements of $\Pic(\ms X)$ must have  pure weight and hence there is a natural map taking a sheaf to its character:
  \[
    \Pic(\ms X)\to\widehat \G
  \]
Since the group law on $\widehat \G$ is determined by
 tensoring characters together, this morphism is a group homomorphism.

\begin{lem}\label{lem:det-tw}
  Given a perfect complex $P\in D^{(a)}(\ms Y)$ of rank $r$, the determinant
  sheaf $\det(P)$ lies in $\Pic^{(ra)}(\ms Y)\subset\Pic(\ms Y)$.
\end{lem}
\begin{proof}
  It suffices to prove this after passing to a geometric fiber of $\ms Y\to Y$,
  so we may assume that $\ms Y=\B\G_{K}$ for some algebraically closed field
  $K$. A perfect complex of $a$-fold twisted sheaves of rank $r$ is then
  identified with a graded $\G$-representation that has graded rank $r$ and
  whose underlying $\G$-representation is isotypic of type $a$. Since the
  determinant is computed as the alternating tensor product of the wedge powers
  of the factors, we see that the type of the determinant is precisely $ra$, as
  desired. 
\end{proof}

\subsection{Twisted kernels}
\label{sec:twisted-derived-equiv}

By \cite{twisted_kernel}, a large class of exact functors $D(X,\alpha)\to D(Y,\beta)$, including equivalences, are given by Fourier-Mukai transforms. In the classical literature on twisted Fourier-Mukai functors, the kernel is given as an object in 
$D(X\times Y,\alpha^{-1}\boxtimes\beta)$, where 
$\boxtimes$ denotes the tensor product of the pullbacks of $\alpha^{-1}$ and $\beta$ to the product $X\times Y$.

However, it is \textit{a priori} unclear how, in the gerbe point of view, 
Fourier-Mukai functors
constructed using the objects in 
$D(X\times Y,\alpha^{-1}\boxtimes\beta)$, 
which are $1$-twisted perfect complexes on a $\G$-gerbe,
yield a map from $1$-twisted sheaves on a gerbe associated to $\alpha$  to $1$-twisted sheaves on a gerbe associated to $\beta$.
There is an identification hiding in the background of many treatments of this subject that makes our gerbe approach to twisted Fourier-Mukai kernels equivalent to the above definition, which we take the opportunity to discuss here.

Fix a closed subgroup $\G\subset\G_m$.
Suppose $Y$  is a quasi-compact quasi-separated scheme and $\ms Z\to Y$
and $\ms Z'\to Y$ are $\G$-gerbes. 
Given a quasi-coherent sheaf $F$
on $\ms Z\times_Y\ms Z'$, we may use the weights of the $\G$-action on $p_{\ms Z}^*F$ and $p_{\ms Z'}^*F$ 
to decompose $F$ as follows:
\[
F=\bigoplus_{\widehat \G\times\widehat \G}F_{(i,j)}
\]
Given $i,j\in\widehat \G$ we 
write $\Qcoh^{(i,j)}(\ms Z\times\ms Z')$ for the subcategory of quasi-coherent
sheaves $F$ such that $F=F_{(i,j)}$.

Since any morphism $\G\to\G_m$ factors uniquely through $\G$, we may view $\widehat{\G}$ as $\Hom(\G,\G)$.
Given a pair $(i,j)\in\widehat \G\times\widehat\G$, define a map
\[
\sigma_{(i,j)}:\G\times\G\to\G
\]
by sending $(a,b)$ to $i(a)\cdot j(b)$. 
There is an induced map
$$(\sigma_{(i,j)})_\ast:\H^2(Y,\G\times\G)\to\H^2(Y,\G)$$
that we can realize on the level of stacks (using Giraud's theory \cite{giraud}) as
follows: given a pair of $\G$-gerbes $\ms Z\to Y$ and $\ms Z'\to Y$, there is
a $\G$-gerbe $(\ms Z\wedge_{(i,j)}\ms Z')$ and a morphism
\[
\tau(\ms Z,\ms Z',i,j):\ms Z\times_Y\ms Z'\to (\ms Z\wedge_{(i,j)}\ms Z')
\]
of $Y$-stacks whose induced morphism on inertia is precisely $\sigma_{(i,j)}$.
The morphism is locally modeled on the composition
\[
\B\G\times\B\G\to\B(\G\times\G)\to\B\G,
\]
where the first map is the canonical one and the second is induced by
$\sigma_{(i,j)}$.

\begin{lem}\label{lem:sum}
  Let $\ms Z$ and $\ms Z'$ be $\G$-gerbes on $Y$ and
  $\sigma_{(i,j)}:\G\times\G\to\G$ a morphism as above. Pulling back by
  $\tau(\ms Z,\ms Z',i,j)$ induces an equivalence of categories
  \[
    \Qcoh^{(1)}(\ms Z\wedge_{(i,j)}\ms Z')\to\Qcoh^{(i,j)}(\ms Z\times\ms Z')
  \]
\end{lem}
\begin{proof}
  These categories are both the global sections of stacks on $Y$, so to prove
  the statement it suffices to assume that $\ms Z$ and $\ms Z'$ are both
  isomorphic to $\B\G$, in which case our morphism becomes $\B\sigma_{(i,j)}$.
  The statement then follows since the $(i,j)$-eigensheaf is precisely the sheaf
  with action factoring through $\sigma_{(i,j)}$.
\end{proof}

  Since $X$ and $X'$ are quasi-compact and quasi-separated, we have that the
  bounded derived category of $\ms O$-modules with quasi-coherent cohomology is
  equivalent to the bounded derived category of quasi-coherent sheaves. 
Thus we have the following corollary:

\begin{cor}\label{cor:kernels}
  With the above notation, we have that $\L\sigma_{(i,j)}^\ast$ identifies the
  derived 
  category of 
  perfect complexes of $\ms Z\wedge_{(i,j)}\ms Z'$-twisted sheaves with the derived
  category of perfect complexes on $\ms Z\times_Y\ms Z'$ of type $(i,j)$.
\end{cor}

\begin{defn}\label{defn:fm}
  Given two $\G$-gerbes $\ms X\to X$ and $\ms Y\to Y$, a \emph{twisted
    Fourier--Mukai kernel\/} is a perfect complex $P\in D^{(-1,1)}(\ms
  X\times\ms Y)$. The \emph{Fourier--Mukai transform\/} associated to $P$ is
  the functor
 \[
   \Phi_P:D^{(1)}(\ms X)\to D^{(1)}(\ms Y)
 \]
 defined by
 \[
    \Phi_P(a)=\R(\pr_2)_\ast\bigl( \mathbf{L}\pr_1^\ast a\ltensor P \bigr)
  \]
\end{defn}

By Corollary~\ref{cor:kernels}, this  definition is equivalent to the one discussed at the beginning of the section.

These Fourier-Mukai equivalences commute with Serre functors, hence we have the following result:

\begin{lem}\label{lem:canonical}
  If there is a Fourier--Mukai equivalence $\Phi_P:D^{(1)}(X,\alpha)\to
  D^{(1)}(Y,\beta)$ then $\dim X=\dim Y$ and $[\omega_X]\in\Pic(X)$ and
  $[\omega_Y]\in\Pic(Y)$ have the same order.
\end{lem}

\subsection{Twisted Mukai vectors and action on cohomology}
\label{sec:mukai-vec}
In this section we briefly consider the theory of twisted Mukai vectors, which
will be essential for showing the vanishing of certain obstructions in Section
\ref{sec:two-lemm-deform} below.

Let $X$ be a smooth projective variety over a field $k$ and $\pi:\ms X\to X$ be a $\m_n$-gerbe for some $n$ invertible in $k$. Since $\pi$ is
the coarse moduli map of an  Deligne--Mumford stack, there is an induced
isomorphism on Chow theory (with respect to rational equivalence) with rational
coefficients  \cite[Theorem~0.5]{Gillet}:
\[
\pi_\ast: A^\ast(\ms X)_\Q\to A^\ast(X)_\Q.
\]
We can use this to define a Chern character and Mukai vector for any perfect
complex of $\ms X$-twisted sheaves, as in \cite[\S 4.1]{twistor}.

\begin{defn}
  In this setting, $\ms X\to X$ has the resolution property, so we may define the Chern character of
   a perfect complex of $\ms X$-twisted sheaves by extending the following:
  \begin{enumerate}
  \item Given a locally free $\ms X$-twisted sheaf  $\E$, the \emph{Chern character\/} of $P$ is the class
    \[\chern(\E)\coloneqq\chern(\E^{\tensor n})^{1/n}\in A^\ast(X)_\Q,\]
    chosen so that its component in $A^0(X)_\Q$ has the same sign as the rank of
    $\E$;
  \item The \emph{Mukai vector\/} of $P$ is the class
    \[v(P)\coloneqq\chern(P)\cdot\sqrt{\Td_X}\in A^\ast(X)_\Q.\]
  \end{enumerate}
\end{defn}

\begin{remark}
This definition agrees over $\C$ with the twisted Chern character defined by Huybrechts in \cite[\S 2.1]{isogenous}. It differs by a factor from the definition introduced in \cite{huybrechtsstellari} (see  \cite[\S 1.3]{isogenous} for a comparison), which is more noncanonical as it depends on a choice of $B$-field, but has the advantage of integrality.
\end{remark}

Now suppose that $\ms X\to X$ and $\ms Y\to Y$ are $\m_n$-gerbes over smooth
projective varieties, and let $P\in D^{(-1,1)}(\ms X\times\ms Y)$ be a
twisted Fourier--Mukai kernel.  
Its Mukai vector can be used to define the following map:
\begin{align*}
  &\Psi_{v(P)}:A^\ast(X)_\Q\to A^\ast(Y)_\Q
  \\
&\Psi_{v(P)}(x)=(\pr_2)_\ast(\pr_1^\ast x\cup v(P)).
\end{align*}
The following is a well-known compatibility between the various transforms.
\begin{lem}\label{lem:mukai-fn}
  For any $a\in D^{(1)}(\ms X)$ we have 
\[
\Psi_{v(P)}(v(a))=v(\Phi_P(a))
\]
as elements of $A^\ast(Y)_\Q$.
\end{lem}
\begin{proof}
  The proof in the untwisted case is a simple application of the
  Grothendieck--Riemann--Roch formula. The twisted case is a bit more subtle; a
  proof can be found in \cite{toen_rr} (cf., for instance, \cite[Appendix~A]{orb_top}). 
\end{proof}

\begin{cor}\label{cor:fm-muk}
  Suppose $\omega_{Y}$ is torsion.
If  $a,b\in D^{(1)}(\ms X)$ satisfy $v(a)=v(b)$, then 
\[
\det(\Phi_P(a))=\det(\Phi_P(b))\in\Pic(\ms Y)\tensor\Q=\Pic(Y)\tensor\Q.
\]
\end{cor}
\begin{proof}
  From the formula for the Todd class, we have that \[
    \Td_{Y}=1-\frac{1}{2}c_1(\omega_{Y})+\cdots \quad\text{and hence}\quad
    \sqrt{\Td_{Y}}=1-\frac{1}{4}c_1(\omega_{Y})+\cdots.
\]
It follows that the component of $v(\Phi_P(a))$ in $A^1(Y)_\Q$ is \[
v_1(\Phi_P(a))=\det(\Phi_P(a))-\frac{1}{4}\rk(\Phi_P(a))c_1(\omega_{Y}).
\]
The result now follows from the fact that $\omega_{Y}$ is torsion, so that
$c_1(\omega_{Y})=0\in A^\ast(Y)_\Q$. 
\end{proof}

Applying the usual cycle class maps to $v(P)$ also defines induced
correspondences on rational cohomology theories. (See \cite[\S 2]{LO} for
more information.) In particular, we have the following.

\begin{lem}\label{lem:coho-iso}
  Suppose $H$ is a Weil cohomology theory and $(X,\alpha)$ and $(Y,\beta)$ are
  two smooth projective twisted varieties. Any derived equivalence 
  $$\Phi_{P}: D^{(1)}(X,\alpha)\rightarrow D^{(1)}(Y,\beta)$$
induces isomorphisms
$$\Phi_P^{\text{\rm odd}}: H^{\text{\rm odd}}(X)\rightarrow 
H^{\text{\rm odd}}(Y)\quad\text{and}\quad\Phi_P^{\text{\rm even}}: H^{\text{\rm even}}(X)\rightarrow
 H^{\text{\rm even}}(Y).$$
\end{lem}
\begin{proof}
  See \cite[Lemma 3.1]{kh} for details.
\end{proof}

\subsection{Brauer groups in families}
\label{sec:brauer-sheaf}

In this section we briefly discuss the variation of the Brauer group
in families of Enriques and bielliptic surfaces. We recall that, for a scheme $X$, the \emph{Brauer group} of $X$, $\operatorname{Br}(X)$ is the group of Azumaya algebras over $X$ up to Morita equivalence. The \emph{cohomological Brauer group} of $X$ is $\H^2_{\text{{\'e}t}}(X,\mathbf{G}_m)_{\text{tors}}$. There is a canonical inclusion $\operatorname{Br}(X)\hookrightarrow\H^2_{\text{{\'e}t}}(X,\mathbf{G}_m)_{\text{tors}}$, which is an isomorphism whenever $X$ is quasi-compact and separated. The reader is referred to \cite{grothBrauerI} and \cite{dejongGabber} for details. Under these assumptions we will denote both the Brauer group and the cohomological Brauer group by $\operatorname{Br}(X)$.

Suppose $\pi: X\to S$ is a proper morphism 
that is cohomologically flat in dimension $0$ (see \cite[6.1.2]{EGAIII2}). Fix a positive
integer $n$ that is invertible on $S$. We assume that the Brauer group and cohomological Brauer group coincide for each geometric fiber of $\pi$; this holds, for example, if $\pi$ is locally projective or if the fibers are smooth of dimension $2$.

Let $\BrSh_{X/S}$ be the sheaf
associated to the presheaf that sends $T\to S$ to the Brauer group
$\Br(X_T)$. There is a natural injective map of
sheaves $$\BrSh_{X/S}\to\R^2\pi_\ast\G_m$$
on the big \'etale site of $S$. This induces an injection
$$\BrSh_{X/S}[n]\to\R^2\pi_\ast\G_m[n].$$

\begin{prop}\label{prop:brauer-var}
  Suppose that for all $T\to S$ we have
  $\R^2(\pi_T)_\ast\ms O_{X_T}=0.$
  Then the sheaf $\BrSh_{X/S}[n]$ is represented by a finite \'etale
  cover of $S$. In particular, if $S$ is strictly Henselian then
  restriction defines isomorphisms
  $$\Br(X_K)'\leftarrow\Br(X)'\to\Br(X_k)',$$
  where $k$ is the residue field at the closed point of $S$, $K$ is
  an algebraic closure of the function field of $S$, and $\Br()'$
  denotes the subgroup of the Brauer group consisting of classes with
  order prime to the characteristic exponent of $k$.
\end{prop}
\begin{proof}
  Consider the short exact sequence $$0\to\Pic_{X/S}\tensor\Z/n\Z\to\R^2\pi_\ast\m_n\to\R^2\pi_\ast\G_m[n]\to
0.$$ Since $\H^2(X_s,\ms O_{X_s})$ vanishes in every fiber, we see
that $\Pic_{X/S}$ formally smooth, whence $\Pic_{X/S}\tensor\Z/n\Z$ is
locally constant (as it is an algebraic space that is finite and \'etale
over the base). The proper and smooth base change theorems in \'etale
cohomology tell us the same thing about $\R^2\pi_\ast\m_n$. It follows
that the quotient
$$Q\coloneqq\R^2\pi_\ast\m_n/\Pic_{X/S}\tensor\Z/n\Z$$
is a locally constant sheaf of finite abelian groups.

By assumption, for each geometric fiber of $X/S$ the inclusion
$$\BrSh_{X/S}[n]\subset Q$$ is surjective on each geometric fiber,
hence is an isomorphism of sheaves. Applying this for all $n$ invertible in $S$ yields
the desired result. 
\end{proof}

\subsection{Deformations of twisted schemes}
\label{sec:defos-tw-sch}\label{sec.th}

\begin{defn}
Given a flat twisted space $(X,\alpha)$ over a scheme $S$ and a closed
immersion $S\to S'$, a \emph{deformation of $(X,\alpha)$ over $S'$\/}
is a flat twisted space $(X',\alpha')$ over $S'$ together with an
isomorphism $X\to X'_S$ such that $\alpha'_X=\alpha$.
\end{defn}

Given an algebraically closed field $k$, let $W(k)$ denote the Witt
vectors of $k$.

\begin{defn}\label{defn:art}
  The category $\Art_W^k$ is the category of Artinian augmented
  $W$-algebras: an object is a local Artinian $W$-algebra $W\to (A,\mathfrak
  m)$ together with an identification $k\to A/\mathfrak m$.
\end{defn}
\begin{defn}
  Suppose $k$ is algebraically closed.
  Given a twisted $k$-variety $(X,\alpha)$, the \emph{deformation
    functor of $(X,\alpha)$\/} is the functor
  $$\Def_{(X,\alpha)}: \Art_W^k\to\Set$$
  that sends $A$ to the set of isomorphism classes of deformations of
  $(X,\alpha)$ over $\Spec A$.
\end{defn}

\begin{cor}\label{cor:defs}
  If $X$ is a smooth projective surface over $k$ such that $\H^2(X,\ms
  O_X)=0$ then for any $\alpha\in\Br(X)$ 
 with
  order prime to the characteristic exponent of $k$
the forgetful morphism
  $$\Def_{(X,\alpha)}\to\Def_X$$
  is an isomorphism.
\end{cor}
\begin{proof}
  This follows from Proposition \ref{prop:brauer-var}, since the sheaf
  of Brauer groups is finite \'etale, hence splits over strictly
  Henselian rings.
\end{proof}

\subsection{Equivariant lifts of derived equivalences}
\label{sec:relat-bridg-maci}

Fix a base scheme $S=\Spec(A)$ with $A$ a strictly Henselian local ring, and fix
proper smooth twisted schemes $(X,\alpha)\to S$ and
$(Y,\beta)\to S$. Suppose the relative dualizing sheaves $\omega_{X/S}$ and
$\omega_{Y/S}$ are torsion of order $n$, with $n$ invertible on
$S$. Given  trivializations $\smash{\omega_{X/S}^{\tensor n}}\iso
\oO_X$ and $\smash{\omega_{Y/S}^{\tensor n}}\iso\oO_Y$, let $\smash{\widetilde X}\to X$
and $\smash{\widetilde Y}\to Y$ be the finite \'etale covers corresponding to
the associated classes in $\H^1(X,\m_n)$ and $\H^1(Y,\m_n)$,
respectively (see \cite[\href{https://stacks.math.columbia.edu/tag/03PK}{Tag 03PK}]{stacks-project} for a discussion of this correspondence). 

Bridgeland and Maciocia \cite{bridgequotient} studied the problem of
lifting derived equivalences $D(X)\to D(Y)$ to
$D(\widetilde X)\to D(\widetilde Y)$ when $S=\Spec\C$ and the twisting classes
are trivial.  In
\cite[Proposition~2.1]{addingtonwray}, the authors extend this work to
lifting equivalences between derived categories of coherent twisted 
sheaves, also over $\C$.
The argument is given in the particular
setting of Enriques surfaces, where $\omega_X^{\otimes 2}\cong \oO_X$,
but it translates immediately to the general case of smooth,
projective varieties with torsion canonical bundle.

When $n$ is  invertible in the base, the methods of \cite{bridgequotient,addingtonwray} can be adapted to prove a relative form of these results. With this assumption, the proofs translate word-for-word. However, since there are some subtleties in checking this that are discussed elsewhere but in different sources, we assemble a discussion or references for all of them for the reader's convenience.

\begin{prop}\label{BMRelative} 
Let $\ms X\to X$ and $\ms Y\to Y$ be $\m_n$-gerbes 
representing $\alpha$ and $\beta$, and we write $\widetilde{\ms X}\coloneqq \ms
X\times_X \widetilde X$ and $\widetilde{\ms Y}\coloneqq \ms Y\times_Y\widetilde Y$.
Given a Fourier--Mukai equivalence 
\[\Phi_{P}: D^{(1)}(\ms X)\rightarrow D^{(1)}(\ms Y)\]
with kernel \[P\in D^{(-1,1)}(\ms X\times_S\ms Y),\] there
exists \[\widetilde{P}\in D^{(-1,1)}(\widetilde{\ms X}\times_S\widetilde{\ms Y})\] so that 
\[\Phi_{\widetilde{{P}}}: D^{(1)}(\widetilde{\ms X})\rightarrow
D^{(1)}(\widetilde{\ms Y})\]
is an equivalence and the following diagrams commute:
  \begin{equation}\label{diag} 
\vcenter{
\xymatrix{ 
D^{(1)}(\widetilde{\ms X})\ar[d]_{\pi_{\ms X*}}
\ar[rr]^{\Phi_{\widetilde{{P}}}}
&&
D^{(1)}(\widetilde{\ms Y})\ar[d]^{\pi_{\ms Y*}}
\\
D^{(1)}(\ms X)\ar[rr]_{\Phi_{{{P}}}}
&&
D^{(1)}(\ms Y)
}
}
\quad\quad
\vcenter{
\xymatrix{ D^{(1)}(\widetilde{
\ms X})\ar[rr]^{\Phi_{\widetilde{{P}}}}&&D^{(1)}(\widetilde{\ms 
Y})
\\
D^{(1)}(\ms X)\ar[u]^{\pi_{\ms X}^*}\ar[rr]_{\Phi_{{{P}}}}&&D^{(1)}(\ms Y)\ar[u]_{\pi_{\ms Y}^*}
}}
\end{equation}
\end{prop}

An equivalence  $\Phi_{\widetilde{{P}}}$ making the diagrams
\eqref{diag} commute is called \textit{lift} of $\Phi_{{P}}$.

\begin{proof}
The methods rely
on the following result about equivariant objects for cyclic coverings. Suppose $Z$ is a 
Deligne--Mumford stack such that every element of $\Gamma(Z,\ms O_Z)$ is an $n$th
power (for example, $Z$ is proper over $A$, since $A$ is strictly Henselian and
$n$ is 
invertible on $A$), $L$ is an
invertible sheaf on 
$Z$, and $\sigma: L^{\tensor n}\to\oO_Z$ is an isomorphism. Write
\begin{equation}\label{cover}
  \widetilde Z=\underline\Spec_Z\bigoplus_{i=1}^{n-1}L^{\tensor i}\to Z
\end{equation}
for
the $\m_n$-cover associated to the pair $(L,\sigma)$. 
Given a simple object $E$ of $D_{\Qcoh}(Z)$ (the bounded derived category with
quasi-coherent cohomology), there is
some $\widetilde E\in D_{\Qcoh}(\widetilde Z)$ such that $p_\ast\widetilde
E\cong E$ if and only if there is an isomorphism
$\lambda: E\tensor L \to E$.

To prove this, we first proceed by induction and use
exactness of tensoring with $L$ to reduce to the case that $E$ is a
quasi-coherent sheaf as in \cite[Proposition~2.5]{bridgequotient}.
The isomorphism $\lambda$ induces an action of $\bigoplus_{i=0}^{n-1}L^{\tensor i}$ on $E$. 
It follows from the definition of \eqref{cover} that sheaves with such an action that is furthermore compatible with $\sigma$ are naturally equivalent to pushforwards of $\oO_{\widetilde  Z}$-modules.

There is a subtle point here: 
 the action of $\bigoplus
L^{\tensor i}$ on $E$ must be made compatible with the
$\oO_Z$-module action $\sigma$.
Let $\varphi:E\otimes L^{\otimes n}\to E$ be the action of $L^{\otimes n}$  on $E$ induced by $\lambda$. Consider the composition $(\sigma\otimes E)\circ \varphi^{-1}\in \Aut(E)$. Since we have required $E$ to be simple, we have $\Aut(E)=\End(E)^{\times}=\oO_Z(Z)^{\times}$, and because we have required that all elements of $\oO_Z(Z)$ be $n$th roots, we can take an $n$th root of this composition, call it $\psi$. Replacing $\lambda$ with $\psi\circ\lambda$ yields an isomorphism with the needed compatibility.

The kernel $\widetilde{P}$ is constructed by applying this result to find a lift of $(\pi_X\times\id_Y)^*P$ to $\widetilde{X}\times\widetilde{Y}$; since $P$ is a kernel of a Fourier-Mukai equivalence, it is simple, so furthermore $(\pi_X\times\id_Y)^*P$ is simple;
cf.~\cite[Proposition~2.1]{addingtonwray}. (Note that simplicity is omitted from \cite[Proposition 2.5(b)]{bridgequotient}.) By construction, we have that the diagram in the statement of the theorem commutes.

The argument given in, for instance, \cite[Proposition~7.18]{huybrechts}
shows that the restriction of $\Phi_{\widetilde{{P}}}$ to the residue field of $A$ is an equivalence.
To show that  $\Phi_{\widetilde{{P}}}$ is an equivalence, we appeal to a derived
form of the Nakayama
lemma. That is: $\Phi_{\widetilde{P}}$ has left and right adjoints,
with kernels $\widetilde{L}$ and $\widetilde{R}$.
The adjunctions induce maps on kernels $\widetilde{L}\circ
\widetilde{P}\to \Delta_*\oO_{\ms X}$ and $\Delta_*\oO_{\ms Y} \to
\widetilde{P}\circ \widetilde{R}$.
The restrictions of these morphisms to the residue field of $A$ are
isomorphisms by \cite[Lemma~4.3(a)]{bridgequotient}, whose proof works
over any separably closed field in which $n$ is invertible.
Hence, by Nakayama's Lemma, $\Phi_{\widetilde{P}}$ is an equivalence.
\end{proof}

\section{Deformations of twisted Fourier--Mukai kernels}
\label{sec.fi}

\subsection{Some results on the deformation theory of complexes}
\label{sec:two-lemm-deform}

Let $A\to A_0$ a square-zero extension of rings with kernel $I$.

\begin{lemma}\label{L:obstruction-pulls-back} Let $\gamma: (Z,\alpha)\to (X,\beta)$ be
	an untwisted finite \'etale morphism of flat separated twisted
        schemes over $A$. Let $(Z,\alpha)_{A_0}$ be the twisted scheme $(Z|_{A_0},\alpha|_{A_0})$.  Given a 
	perfect complex $P\in D((Z,\alpha)_{A_0})$, the natural map
        \begin{equation}\label{pullback}
          \Ext^2_{X_0}(P,P\otimes I)\to\Ext^2_{Z_0}(\LL\gamma^\ast P,
          \LL\gamma^\ast P\ltensor I)
\end{equation}
          sends the obstruction to deforming $P$
	over $A$ to the obstruction to deforming $\LL\gamma^\ast P$ over A.
\end{lemma}

\begin{proof} 
	The proof of this result we offer here is undoubtedly far from
	ideal. Unfortunately, there does not appear to be an argument in the
	literature that is general enough to work without resorting to the
	techniques used in \cite[Section 3]{deformingkernels}. The idea there
	is as follows: given the complex $P$, one replaces it by a ``good
	resolution'' $P'\to P$, which is a quasi-isomorphism in which $P'$ has
	terms of the form $\oplus j_!\oO_U$ for {\'e}tale morphisms $j: U\to
	X$ with $U$ affine. The complex $P'$ is itself $K$-flat (in the sense
	of \cite{MR932640}), so it can be used to compute derived tensor
	products. The obstruction class arises by computing good resolutions
	of $P$ over $A_0$ and over $A$ and then making an explicit map 
	$$Q\to P\ltensor_{A_0}I[1]$$ in $D^b(X_0)$, where $Q$ is the homotopy
	limit of the derived adjunction map $P\ltensor_A A_0\to P$. (Details may be
  found in \cite[Construction 3.2.8]{deformingkernels}.) To adapt this to the
  twisted setting, we replace $\ms O$ with an Azumaya algebra $\ms A$
  representing $\beta$. The rest of the arguments carry over verbatim.
	
	Consider
	the Cartesian diagram
	$$\xymatrix{
		U'\ar[r]^{j'}\ar[d]_{\gamma'} & Z\ar[d]^\gamma\\
		U\ar[r]_j & X.
	}$$
	To establish the appropriate functoriality of the obstruction class,
	it suffices to show that $\gamma^\ast j_!\ms A_U=j'_!\ms A_{U'}$.
	By flat base change we have a canonical isomorphism of functors 
	$j^\ast\gamma_\ast=\gamma'_\ast(j')^\ast$, giving rise to a canonical 
	isomorphism $\gamma^\ast 
	j_!=j'_!(\gamma')^\ast$. This
	gives the desired result.
\end{proof}

In what follows, we will use a result that is essentially a folk theorem. We state it here as a conjecture; a proof will appear in \cite{lieblicholsson}.
\begin{conj}\label{L:obstruction-traces-down}
  Let $X\to\Spec A$ be a flat Artin stack such that $\omega_{X/A}\cong \ms O_X$ 
and let $X_0=X\tensor_A A_0$.
  Suppose $P_0$ is a perfect complex on $X_0$ with
  determinant $L_0$. Let \[o(P_0)\in\Ext^2_{X_0}(P_0,P_0\ltensor I)\] and
  \[o(L_0)\in\Ext^2_{X_0}(\ms O, \ms O\ltensor I)\] be the
  obstruction classes to the deformation of $P_0$ and $L_0$ to $X$. Then the
  trace map
\begin{equation}\label{trace}
    \Ext^2_{X_0}(P_0, P_0\ltensor I)\to\Ext^2_{X_0}(\ms O,\ms O\ltensor I)
\end{equation}
  sends $o(P_0)$ to $o(L_0)$.
\end{conj}

\begin{remark}
  Note that Conjecture \ref{L:obstruction-traces-down} is phrased for a flat
  Artin stack, so it doesn't fit into the usual Illusie topos 
  framework. Nevertheless, the deformation-obstruction theory for perfect
  complexes has the same formal propertiers as in the classical topos-theoretic
  case by the alternative argument of Grothendieck \cite[IV.3.1.12]{IllusieLuc1971Cced}.
\end{remark}
\begin{proof}[Remarks on the state of Conjecture \ref{L:obstruction-traces-down}.]
	As far as we can tell, there is no proof anywhere in the literature even in
	the case of deformations of a coherent sheaf. The proof for locally free sheaves on projective varieties is proved in \cite[Theorem 3.23]{thomas2000}, but  
	assumptions are made about the base schemes (see the conditions immediately
	preceding 
	equation 3.8) that do not apply to general deformation
	problems, including lifting (for example, lifting from $\Z/p\Z$ to
	$\Z/p^2\Z$). There is a derived version of the statement 
	in \cite[Section 3]{MR3341464}, but [ibid.] works over $\C$ rather than a
	general base. There is another derived version at \cite{MR3701353} that also assumes the base has characteristic $0$. Finally, there is a similar statement at the end of
	\cite{MR3103888}, which works over general Noetherian bases, but there are
	flatness assumptions that do not hold in the present context (or, more
	generally, in the context of infinitesimal deformation problems where the base
	is not a field). We will not attempt to give a proof here, so we will refer to
	this statement as a conjecture rather than a lemma. A write-up that works in sufficient generality will appear in \cite{lieblicholsson}.
\end{proof}

\begin{lem}\label{L:pullback-thang}
	Suppose $\gamma: (Z,\alpha)\to (X,\beta)$ is an untwisted 
	finite \'etale 
morphism of smooth proper twisted schemes of relative
        dimension $2$ over $A$ such that 
	$\omega_{Z/A}\cong\ms O_Z$ and its degree is invertible in $A$. 
	Suppose $Q$ is a
	perfect complex on $(X_{A_0},\beta_0)$ such that
\renewcommand\theenumi{\arabic{enumi}}
	\begin{enumerate}
		\item the invertible sheaf $\det Q\in\Pic(Z,\alpha)$ is unobstructed with
      respect to the  
		extension $(Z_0,\alpha_0)\subset(Z,\alpha)$;
		\item we have $$\LL \gamma^\ast Q\cong\bigoplus_{i=1}^m Q_i$$ for some 
		$m$ invertible in
		$A_0$ with each $Q_i$ a simple perfect complex on $(Z_0,\alpha_0)$;
		\item for each $i=2,\ldots,m$ we have 
		$$\det Q_i\cong\det Q_1\tensor\Lambda_i$$ for some 
		$\Lambda_i\in\Pic(Z_0,\alpha_0)$ 
		that is
		unobstructed with respect to the extension $(Z_{0},\alpha_0)\subset(Z,\alpha)$.
	\end{enumerate}
	Then the complex $Q$ is unobstructed.
\end{lem}

\begin{proof}
  Since the degree $d$ of $\gamma$ is invertible in $A$, the map $\frac{1}{d}\tr_{\gamma}$ is a section of $O_X\to\gamma_*\oO_Z$. Hence, \eqref{pullback} is injective and by
Lemma~\ref{L:obstruction-pulls-back}, it suffices to show that $\LL 
	\gamma^\ast Q$ is unobstructed.
        Further by 
	Conjecture~\ref{L:obstruction-traces-down} it suffices to show that $\det(Q_1)$ 
	is unobstructed; note that
        the simplicity assumption is necessary for \eqref{trace} to be injective in the relevant case.
        
	By assumption,
	$$\LL \gamma^\ast Q\cong \det(Q_1)^{\otimes m}\otimes 
	\Lambda_2\otimes\cdots\otimes\Lambda_m .$$
	Since the $\Lambda_i$ and $\det(Q)$, and hence $\gamma^*\det(Q)$ are 
	unobstructed, we have $m\ob(\det Q_1)=0$.
	Since $m$ is invertible in $A_0$, we conclude that $\ob(\det Q_1)=0$.
\end{proof}

The following setting puts us in the situation of either Theorem~\ref{thm:bielliptic} or \ref{thm:enriques}.

\begin{setting}\label{setting}
Now suppose that $(X,\alpha)$ and $(Y,\beta)$ are twisted schemes that are
either
smooth proper Enriques or 
trivially twisted bielliptic surfaces over an Artinian ring $A$ with
algebraically closed residue field $k$.  
Let $n$ be the order of $\omega_X$ in $\Pic(X)$. 
Suppose $$\Phi_P: D(X,\alpha)\to D(Y,\beta)$$ is a Fourier--Mukai equivalence. 
We assume that the characteristic
of $k$ is at least 3 if $X$ is Enriques and at least $5$ if $X$ is bielliptic.

Note that in this setting, $\m_n$-gerbes representing $X$ and $Y$ may be chosen as the Brauer group of an Enriques surface has a cardinality of $2$.
\end{setting}

\begin{lem}\label{L:pullback-bi}
Suppose we are in the situation described in Setting~\ref{setting}.
Given a section $x\in X(A)$, let $\ms L_x$ denote an
  invertible $\alpha|_x$-twisted sheaf. There is a decomposition
	$$\LL\pi_Y^\ast\Phi_P(\ms L_x)\cong\bigoplus_{i=1}^n Q_i$$
	in $D(\widetilde Y,\beta_{\widetilde Y})$ that satisfies the 
	conditions 
	of Lemma~\ref{L:pullback-thang}.
\end{lem}
\begin{proof}
  Choose a finite closed subgroup $\G(=\m_n)\subset\G_m$ and $\G$-gerbes $\ms X\to X$
  and $\ms Y\to Y$ representing $\alpha$ and $\beta$. Choosing such structures
  allows us to use the theory of twisted Mukai vectors. 
	Let 
	$$\Phi_{\widetilde
    P}: D(\widetilde{\ms X})\to D(\widetilde{\ms Y})$$ be the canonical covering
  equivalence induced by $\Phi_P$ as described in Section
  \ref{sec:relat-bridg-maci}. By the commutativity of diagrams \eqref{diag},
  there is an isomorphism
	$$\LL\pi_Y^\ast\Phi_P(\ms L_x)\cong\Phi_{\widetilde P}(\LL\pi_X^\ast(\ms 
	L_x)).$$ On the other hand, $\LL\pi_X^\ast(\ms L_x)$ can be written as $\ms
  L_{x_1}\oplus\cdots\oplus \ms L_{x_n}$, where
  $\{x_1,\ldots,x_n\}=\pi_X^{-1}(x)$ and $\ms L_{x_i}$ is an invertible
  $\alpha_{x_i}$-twisted sheaf. Note that $\ms
  L_{x_1}|_{\ms X_k},\ldots,\ms L_{x_n}|_{\ms X_k}$ all have the same Mukai
  vector. It follows from Lemma \ref{lem:mukai-fn}
	$$\Phi_{\widetilde P_k}(\ms L_{x_1}|_{\widetilde{\ms X}_k}),\ldots,\Phi_{\widetilde P_k}(\ms 
	L_{x_n}|_{\widetilde{\ms X}_k})\in D(\widetilde{\ms Y}_k)$$ have the same Mukai
  vectors (where $\widetilde P_k$ denotes the derived restriction of $\widetilde
  P$ to $\widetilde{\ms X}_k\times\widetilde{\ms Y}_k$). In particular, 
  the determinants of the complexes $\Phi_{\widetilde P_k}(\ms
  L_{x_i}|_{\widetilde{\ms X}_k})$ are equal in $\Pic(\ms Y_k)\tensor\Q$. It
  follows that for every $i$ we have that the invertible sheaf
  \[
    \Lambda_i\coloneqq\det(\Phi_{\widetilde P}(\ms
  L_{x_i}))\tensor\det(\Phi_{\widetilde P}(\ms
  L_{x_1}))^\vee
  \]
 defines a section of $\Pic^0_{\widetilde{\ms Y}/A}$ over $A_0$.

  Since $Y$ is either Enriques or bielliptic (so that $\widetilde Y$ is either
  K3 or abelian), the scheme $\Pic^0_{\widetilde{\ms Y}/A}$ is smooth over $A$. The
  hypothesis of Lemma~\ref{L:pullback-thang} is thus satisfied.
\end{proof}

\subsection{Moduli of perfect complexes of twisted sheaves}
\label{sec:moduli-perf}

In this section we fix a smooth projective family $\pi: (Y,\alpha)\to S$ of
twisted schemes. We will write $\pPerf_{(Y,\alpha)/S}$ for the stack
of simple universally gluable relatively perfect complexes of
$\alpha$-twisted sheaves, generalizing the notions of \cite{deformingkernels};
see also \cite[\S 3]{1711.00846} for a discussion of this stack.
The
following proposition summarizes the main properties of $\pPerf_{(Y,\alpha)/S}$.

\begin{prop}
  Given $(Y,\alpha)\to S$, $\ms A$, and $\ms Y$ as above, the following hold.
  \begin{enumerate}
  \item The stack $\pPerf_{(Y,\alpha)/S}\to S$ is an Artin stack
    locally of finite presentation with inertia stack $\G_m$. (In particular, it
    is a $\G_m$-gerbe over an algebraic space locally of finite presentation.)
  \item Given a deformation situation $A'\to A\to A_0$ with kernel $I$ (in the
    notation of Artin \cite{artin}), 
    the natural obstruction theory for an object
    $P\in\pPerf_{(Y,\alpha)/S}(A)$ takes values in $\Ext^2_{\ms
      Y_A}(P,P\ltensor I)$. The tangent theory takes values in $\Ext^1_{\ms
      Y_A}(P,P\ltensor I)$.
  \item The determinant defines a morphism
    $$\delta: \pPerf_{(Y,\alpha)/S}\to\sPic_{(Y,\alpha)/S}$$
    of algebraic stacks.
  \end{enumerate}
\end{prop}
\begin{proof}
  There are several proofs of these statements. A low-technology version is to
  realize perfect complexes of twisted sheaves as complexes of $\ms A$-modules
  for an Azumaya algebra $\ms A$ with Brauer class $\alpha$, and note that the
  proofs of \cite{deformingkernels} carry over \emph{mutatis mutandis\/} for
  $\ms A$-modules. A higher-tech version is to import the derived techniques of
  To\"en--Vaqui\'e \cite{MR2493386}, as described in \cite{1711.00846}. We will not
  discuss the details here.
\end{proof}

\begin{remark}
In the twisted case, we also have the following. Proofs are identical
to the proofs in Section 3.1 and Section 4 of \cite{LO}, \emph{mutatis
  mutandis\/}. Given a pair of $\G_m$-gerbes $\ms X\to X$ and $\ms
Y\to Y$ flat and of finite presentation over a base $B$, a
Fourier--Mukai kernel $P\in D^{(-1,1)}(\ms X\times \ms Y)$ 
induces a morphism
\[\kappa_P:\ms X\to\pPerf_{\ms Y/B}.\]
\end{remark}

\begin{prop}\label{prop:open-imm}
  Suppose $\ms X\to X$ and $\ms Y\to Y$ are $\G_m$-gerbes representing
  smooth twisted (relative) varieties over a base scheme $B$. Given a
  Fourier--Mukai equivalence 
  \[\Phi_P:D^{(1)}(\ms X)\to D^{(1)}(\ms Y),\]
  the induced morphism
  \[\kappa_P:\ms X\to\pPerf_{\ms Y/B}\]
  is an open immersion.
\end{prop}
\begin{proof}
  See Proposition 4.4 of \cite{LO}. Note that in \emph{loc. cit.\/}
  the morphism considered is the map from $X$ to the sheafification of
  $\pPerf_{\ms Y}$, and the base is assumed to be a field. The
  proposition claimed here follows from that result (readily modified for
  twisted sheaves) by Nakayama's lemma combined with the fact that the
  induced morphism of inertia \[\G_{m,\ms
    X}\to\kappa^\ast\G_{m,\pPerf_{\ms Y/B}}\] is an isomorphism because
    the functor $\Phi_P$ is linear.
\end{proof}

\begin{cor}\label{cor:rigid}
  Suppose $A\to A_0$ is a square-zero extension of rings and $\ms X\to
  X$ and $\ms Y\to Y$ are $\G_m$-gerbes representing smooth twisted
  schemes over $A$. Write $\ms X_0\to X_0$ and $\ms Y_0\to Y_0$ for
  the restrictions to $A_0$. Given a relative Fourier--Mukai kernel \[P_0\in
    D^{(-1,1)}(\ms X_0\times_{\Spec A_0}\ms Y_0),\]
  there is at most one object \[P\in D^{(-1,1)}(\ms
    X\times_{\Spec A}\ms Y)\] up to quasi-isomorphism 
  such that $\LL i^\ast P\cong P_0$, where $i:\ms X_0\times\ms
  Y_0\to\ms X\times\ms Y$ is the canonical inclusion.
\end{cor}
\begin{proof}
  Since $\kappa_{P_0}$ is an open immersion, it has at most one
  extension to an open immersion $\ms X\to\pPerf_{\ms Y/A}$, up to
  $2$-isomorphism.
\end{proof}

\subsection{An isomorphism of deformation functors}
\label{sec:statements-theorems}

Suppose $k$ is algebraically closed of characteristic at least $3$
(resp.\ $5$) and \[
  \Phi_P:D^{(1)}(X,\alpha)\to D^{(1)}(Y,\beta)
\] is a Fourier--Mukai equivalence. 

\begin{lem}\label{lem:it-is-what-it-is}
  If $X$ is Enriques (resp.\ bielliptic) then $Y$ is Enriques (resp.\
  bielliptic).
\end{lem}
\begin{proof}
  By Lemma \ref{lem:canonical}, $X$ and $Y$ have the same dimension and
  $\omega_X$ and $\omega_Y$ have the same order. Proposition \ref{lem:coho-iso},
  together with Poincar\'e duality, implies that derived equivalent surfaces
  have the same $\ell$-adic Betti numbers. We conclude that $Y$ is Enriques
  (resp.\ bielliptic) using Bombieri-Mumford classification of surfaces in
  positive characteristic \cite{BMII}.
\end{proof}

\begin{assumption}\label{ass:ass}
We assume in the rest of this section that $(X,\alpha)$ and
  $(Y,\beta)$ are twisted Enriques
  (resp. bielliptic) surfaces over an algebraically closed field $k$
  of characteristic at least $3$ (resp. at least $5$). We also fix
  $\G_m$-gerbes $\ms X\to X$ and $\ms Y\to Y$ representing $\alpha$
  and $\beta$. We fix a Fourier--Mukai equivalence
  \[\Phi_P:D^{(1)}(\ms X)\to D^{(1)}(\ms Y).\]
\end{assumption}

Let $A\to A_0$ be a map of Artinian augmented $W(k)$-algebras.
Suppose $\ms X_0\to
X_0$ is a $\G_m$-gerbe representing a point $[\ms
X_0]\in\Def_{(X,\alpha)}(A_0)$ and $\ms Y_A\to Y_A$ is a $\G_m$-gerbe
representing a point $[\ms Y_A]\in\Def_{(Y,\beta)}(A)$. Write $\ms
Y_0\to Y_0$ for the restriction of $\ms Y_A\to Y_A$ to $A_0$.

Suppose
\[\Phi_{P_0}: D^{(1)}(\ms X_0)\to D^{(1)}(\ms Y_0)\] is a relative
Fourier--Mukai equivalence. There is an 
associated open immersion
\[\kappa:\ms X_{A_0}\hookrightarrow\pPerf_{\ms Y_0/A_0}\]
of Artin stacks.
\begin{prop}\label{P:imsosmooth}
  The morphism $\kappa$ has image contained in the smooth locus
  of the structure morphism \[\pPerf_{\ms Y_A/A}\to\Spec A.\]
\end{prop}
\begin{proof}
  Since the smooth locus is open, it suffices to prove the result
  under the assumption that $A_0=k$, and then it is enough to prove
  that for any $x\in X(k)$ and any invertible twisted sheaf $\ms L_x$
  supported on $x$, the image complex $\Phi(\ms L_x)\in D(Y,\beta)$
  lies in the smooth locus of $\pPerf_{\ms Y_A/A}$. This follows
  immediately from Lemma \ref{L:pullback-bi}.
\end{proof}

\begin{thm}\label{T:defo-kernels}
Under Assumption \ref{ass:ass}, there is
an isomorphism of formal deformation functors
$$\rho: \Def_{(Y,\beta)}\to\Def_{(X,\alpha)}$$
with the following property.
Given a point \[[\ms Y_A\to Y_A]\in\Def_{\ms Y}(A)\]
with image \[[\ms X_A\to X_A]=\rho([\ms Y_A\to Y_A])\in\Def_{\ms X}(A),\]
there is a complex \[P\in D^{(-1,1)}(\ms X\times_{\Spec A}\ms Y),\]
unique up to quasi-isomorphism, 
such that  
$\LL i^\ast P_A\cong
P$, where
\[i:\ms X\times_{\Spec k}\ms Y\to\ms X_A\times_{\Spec A}\ms Y_A\]
is the natural closed immersion.
\end{thm}
\begin{proof}
  Given $\ms Y_A\to Y_A$, Proposition \ref{prop:open-imm} and
  Proposition~\ref{P:imsosmooth} imply that $\kappa_P$ identifies $\ms X$ with
  an open substack of $\pPerf_{Y/k}$ that lies in the smooth locus of
  $\pPerf_{\ms Y_A/A}$. Since open substacks of the smooth locus lift uniquely,
  we get an induced deformation $\ms X_A\to X_A$ that is a $\G_m$-gerbe, 
  giving a point of $\Def_{\ms X}(A)$. Restricting the universal complex gives
  the desired kernel $P_A$. (That $P_A$ also gives an equivalence follows from
  Nakayama's Lemma as in \cite[Theorem 6.1]{LO}.)
\end{proof}

By Proposition \ref{prop:brauer-var}, for any strictly Henselian local ring $R$
with residue field $k$ and any lift $Y_R/R$ (resp. $X_R/R$), restriction defines
an isomorphism $\Br(Y_R)\to\Br(Y)$ (resp. $\Br(X_R)\to\Br(X)$). Thus, for any
such lift, there is a canonical way to propogate any twisting class. Given
$\alpha\in\Br(Y)$ and a lift $Y_R/R$, we will write $\alpha_R$ for the canonical
lifting of $\alpha$.  
\begin{thm}\label{T:it-moves} Let $X$, $Y$, and $P$ be as in  Assumption
  \ref{ass:ass}. 
  For any complete local Noetherian ring
$R$ with residue field $k$ and any lift $Y_R\to\Spec R$ of $Y/k$, there is a
lift 
$X_R\to \Spec R$ such that $\Phi_P$ lifts to a relative Fourier--Mukai
equivalence \[\Phi_{P_R}:  
D^{(1)}(X_R,\alpha_R)\to D^{(1)}(Y_R,\beta_R).\]
\end{thm}

\begin{proof}
Given a deformation $Y_R$ of $Y$, from Theorem~\ref{T:defo-kernels} we
get an induced \emph{formal\/} deformation $\mathfrak X_R\in\Def_X(\Spf R)$ and \[\mathfrak P_R\in D^{(1)}(\mathfrak X_R\times_{\Spf
R}\widehat{Y}_R, \alpha_R^{-1}\boxtimes\beta_R).\] Since $\H^2(X,\oO)=0$, any
ample invertible 
sheaf on $X$ lifts to $\mathfrak X_R$, so we can algebraize $\mathfrak
X_R$ to the completion of a relative Enriques surface 
$X_R$ that carries the canonical Brauer class $\alpha_R$ algebraizing the formal
lift. The Grothendieck Existence Theorem for perfect complexes
\cite[Proposition~3.6.1]{deformingkernels} algebraizes $\mathfrak P_R$
to a complex \[P_R\in D(X_R\times_{\Spec R} Y_R,
  \alpha_R^{-1}\boxtimes\beta_R).\] As in \cite[Proof
of Theorem~6.1]{LO}, Nakayama's Lemma then implies that $P_R$ is a
relative twisted Fourier--Mukai equivalence, as desired.
\end{proof}

\begin{remark}
In using the Grothendieck Existence Theorem for perfect complexes in the above proof, we are observing that the methods of \cite{deformingkernels} may be extended to the twisted case. The reader may find it interesting to note that, following the appearance of this paper, Ben Lim has proved a significant strenghtening of this result to perfect complexes on algebraic stacks \cite{benlim}, using methods distinct from \cite{deformingkernels}.
\end{remark}

\section{Proof of Theorem \ref{thm:enriques} and Theorem
  \ref{thm:bielliptic}}
\label{sec:proofs-theorems}
\label{sec.fo} 

Let $F: D^{(1)}(X,\alpha)\to D^{(1)}(Y,\beta)$ be an equivalence 
as in the statement of Theorem \ref{thm:enriques} or
  \ref{thm:bielliptic}.
By
\cite[Theorem 1.1]{twisted_kernel}, there is a kernel $P\in D(X\times Y,
\alpha^{-1}\boxtimes\beta)$ so that $F$ is naturally isomorphic to the
Fourier--Mukai equivalence $\Phi_P$. Enriques (resp.~bielliptic) surfaces of characteristic at least $3$ (resp.~$5$) can be lifted to
a finite extension of the Witt vectors; see \cite[Proposition~6.1.1]{IllusieFGA}
and Partsch \cite{Partsch}. By Theorem~\ref{T:it-moves}, given any lift $Y_R$ of
$Y$ over a finite flat $W(k)$-algebra $R$ there are induced deformations $X_R$
and $P_R\in D(X_R\times_{\Spec R}Y_R, \alpha^{-1}_R\boxtimes\beta_R)$ giving a
relative Fourier--Mukai equivalence. 

By \cite[Proposition~3.6.1]{deformingkernels}, for any ring homomorphism $R\to
S$, the base change $P_S$ induces an equivalence
\[D^{(1)}(X_S,\alpha_S)\to D^{(1)}(Y_S,\beta_S).\] In particular we may choose
an embedding $\kappa(R)\to\C$, yielding a Fourier--Mukai equivalence $$\Phi_{P_{\C}}:
D(X_{\C},\alpha_{\C})\simeq D(Y_{\C},\beta_{\C}).$$ If $\alpha$ and $\beta$ are $0$ then
\cite[Proposition 6.1, 6.2]{bridgesurfaces} implies that $X_{\C}$ and $Y_{\C}$
are isomorphic. If $X$ and $Y$ are Enriques and $\alpha$ and $\beta$ are
arbitrary, then 
\cite{addingtonwray} ensures that $(X_{\C},\alpha_{\C})$ and $(Y_{\C},\beta_{\C})$ are
isomorphic.

Spreading out, we find a finite extension $K'$ of $\kappa(R)$ such that there is
an isomorphism
\begin{equation}\label{eq:iso}
f: (X_{K'},\alpha_{K'})\to(Y_{K'},\beta_{K'})
\end{equation} over $K'$. The
normalization $R'$ of $R$ in $K'$ is a complete DVR with residue field $k$ and
fraction field $K'$. Since invertible sheaves are unobstructed, this isomorphism
preserves relative polarizations over $R'$. Hence, we may use \cite[Theorem
2]{MM} to conclude that the isomorphism $f$ induces an isomorphism $f:
X_{R'}\to Y_{R'}$ of the underlying schemes. By Proposition
\ref{prop:brauer-var}, we see that
$f^\ast\beta_R=\alpha_R$. Specializing to $k$ gives the desired result.

\bibliographystyle{plain}
\bibliography{biblio}

\end{document}